\documentclass[11pt,a4paper]{amsart}
\usepackage{amscd,amsfonts,amsmath,amssymb,amstext,amsthm}
\usepackage[all]{xy}

\newcommand{\RR}{\mathbb{R}} 
\newcommand{\CC}{\mathbb{C}} 
\newcommand{\ZZ}{\mathbb{Z}}

\newcommand{\al}{\alpha}
\newcommand{\be}{\beta}
\newcommand{\ga}{\gamma}

\newtheorem{thm}{Theorem}[section]
\newtheorem{cor}[thm]{Corollary}
\newtheorem{lem}[thm]{Lemma}
\newtheorem{prop}[thm]{Proposition}
\newtheorem{definition}[thm]{Definition}
\newtheorem{remark}[thm]{Remark}
\newtheorem{example}[thm]{Example}
\numberwithin{equation}{section}

\title{A Kobayashi pseudo-distance for holomorphic bracket generating distributions}
\date\today

\author{Aeryeong Seo}
\address{School of Mathematics, Korea Institute for Advanced Study (KIAS),
 85 Hoegiro (Cheongnyangni-dong 207-43), Dongdaemun-gu,
 Seoul 130-722, Republic of Korea}%
\email{aeryeongseo@kias.re.kr}
\thanks{}%

\begin{document}

\maketitle

\markboth{Aeryeong Seo}{Kobayashi pseudo-distance for holomorphic bracket generating  distribution}

\begin{abstract}
In this paper, we generalize the Kobayashi pseudo-distance 
to complex manifolds which admit 
holomorphic bracket generating distributions. 
The generalization is based on Chow's theorem 
in sub-Riemannian geometry. 
Let $G$ be a linear semisimple Lie group.
For a complex $G$-homogeneous manifold $M$
with a $G$-invariant holomorphic bracket generating distribution $D$, 
we prove that $(M,D)$ is Kobayashi hyperbolic if and only if
the universal covering of $M$ is 
a canonical flag domain and the induced distribution is the superhorizontal 
distribution.
\end{abstract}

\section{introduction}
Let $M$ be a complex manifold. 
We say that $D$ is a holomorphic distribution on $M$ 
when $D$ is a holomorphic subbundle of the 
holomorphic tangent bundle $T_M$ of $M$.
We say that it is bracket generating if 
any local holomorphic frame $\{X_1, \ldots, X_d\}$ 
for $D$, together with all of its iterated Lie brackets 
$\{ [X_i, X_j],\, [X_i, [X_j, X_\ell]],\, \ldots\,\}$, spans $T_M$.
In this paper, we generalize the Kobayashi pseudo-distance 
to the complex manifolds
that admit holomorphic bracket generating distributions (HBGD for short) on $M$.

Let us first recall the definition of the Kobayashi pseudo-distance. 
Let $M$ be a complex manifold 
and $\Delta=\{z\in \mathbb C: |z|<1\}$ a unit disc.
For $x, \, y\in M$ define 
$$
\delta_{M}(x,y) = \inf \left \{d_\Delta(a,b): h\colon \Delta\rightarrow M \text{ holomorphic, } \, h(a) = x,\,h(b) = y\right\},
$$
where $d_\Delta$ is the Poincar\'e distance on $\Delta$.
The {\it Kobayashi pseudo-distance} of $M$ is defined by
\begin{equation}
d_{M}(x,y) = \inf  \sum_{j=1}^N\delta_{M}(x_{j-1}, x_j)
\end{equation}
where the infimum is taken over the sets of points $\{x_0,\ldots, x_N\}$ of $M$
such that $x=x_0$ and $y = x_N$. 
When $d_{M}$ is a distance, 
we say that $M$ is {\it Kobayashi hyperbolic}.

In the context of complex manifolds, Kobayashi first introduced it 
as the largest pseudo-distance
on a complex manifold satisfying  
the distance decreasing property with respect to 
the holomorphic mappings between complex manifolds (\cite{Kobayashi}).
If a complex manifold has a complete Hermitian metric
and its holomorphic sectional curvature is bounded from above
by a negative constant, then the manifold is complete Kobayashi hyperbolic.
Because of this property, complete Kobayashi hyperbolic manifolds are considered 
to be a generalization of the Riemann surfaces of genus 
greater than or equal to $2$ equipped with the Poincar\'e distance.


A holomorphic mapping $f\colon N \rightarrow M$ 
between complex manifolds $M,\, N$ is tangential to $D$ 
if $df(T_N)\subset D$. The generalization 
of the Kobayashi pseudo-distance to complex manifolds
with HBGDs in this paper that we propose is the following.

\begin{definition}\label{definition of the Kobayashi distance}
Let $M$ be a complex manifold 
with a holomorphic distribution $D$.
For $x,\, y\in M$, define 
\begin{equation}
\begin{aligned}
\delta_{M,D}(x,y)&= \inf \big\{d_\Delta(a,b): 
h \colon \Delta \rightarrow M \text{ holomorphic }\\
&\quad\quad\quad\quad\quad\quad\quad\text{tangential to } D, \,h(a) = x,\,h(b) = y\big\}
\end{aligned}
\end{equation}
where $d_\Delta$ is the Poincar\'e distance on $\Delta$.
If there is no holomorphic disc tangential to $D$ connecting $x$ and $y$, 
then set $\delta_{M,D}(x,y)=\infty$.
For $x,\,y\in M$, define {\it the Kobayashi pseudo-distance} of $(M,D)$ by
\begin{equation}
d_{M,D}(x,y) = \inf    \sum_{j=1}^N\delta_{M,D}(x_{j-1}, x_j)
\end{equation}
where the infimum is taken over finite sets of points $\{x_0,\ldots, x_N\}$ of $M$
such that $x=x_0$ and $y = x_N$. 
When $d_{M,D}$ is a distance, we say that $(M,D)$ is {\it Kobayashi hyperbolic}.
Moreover, if every Cauchy sequence with respect to 
$d_{M,D}$ has a convergent subsequence, 
then we say that $(M,D)$ is {\it complete Kobayashi hyperbolic}.
\end{definition}
In order to make sense of Definition \ref{definition of the Kobayashi distance}, 
we need $d_{M,D}$ to be bounded on all 
pairs of points belonging to the same connected component. 

\begin{thm}\label{finiteness}
Let $M$ be a connected complex manifold with 
a holomorphic bracket generating distribution $D$. 
Then for any two points $x,\, y$ in $M$, one has $d_{M,D}(x,y)<\infty$.
\end{thm}

The proof of this theorem is based on the one given by Chow in \cite{Chow} 
in the context of sub-Riemannian geometry:
{\it 
let $M$ be a connected manifold and $D$ a bracket generating subbundle 
in the tangent bundle of $M$. 
Then any two points in $M$ can be joined by a horizontal piecewise curve.
}

In \cite{Demailly} the Kobayashi-Royden infinitesimal pseudo-metric
was generalized by Demailly to complex manifolds with holomorphic distributions.
\begin{definition}[Demailly \cite{Demailly}]
The {\it Kobayashi-Royden infinitesimal pseudo-metric} of $(M,D)$ is 
the Finsler metric on $D$ defined for any $x \in M$ and $v\in D_p$ by
\begin{equation}
\begin{aligned}
k_{M,D}(v) &= \inf \big\{ \lambda>0 : f \colon \Delta\rightarrow M  
\text{ holomorphic}\\
& \quad\quad\quad\text{ tangential to } D,  f(0)=x,\, \lambda\, df(0)=v \big\}.
\end{aligned}
\end{equation}
We say that $(M,D)$ is {\it infinitesimally Kobayashi hyperbolic} 
if $k_{M,D}$ is positive definite on every fiber $D_x$ 
and satisfies a uniform lower bound $k_{M,D}(v)\geq  \epsilon |v|_g$ 
for $\epsilon$ sufficiently small and depending on any smooth Hermitian metric 
$g$ on $D$, when $x$ describes a compact subset of $M$.
\end{definition}

If $D$ is the holomorphic tangent bundle of $M$ itself, then $d_{M,D}$ 
and $k_{M,D}$ become the usual Kobayashi pseudo-distance $d_M$ 
and Kobayashi-Royden infinitesimal pseudo-metric $k_M$. 
Notice that, by definition,
\begin{equation}
d_M(x,y)\leq d_{M,D}(x,y), \quad k_M(v)\leq k_{M,D}(v)
\end{equation}
for any $x,\,y\in M$ and $v\in D_x$.
Hence, if a complex manifold $M$ is (infinitesimally) Kobayashi hyperbolic, 
then $(M,D)$ is also (infinitesimally) Kobayashi hyperbolic. 
However, it should be stressed that the converse is not true. 

Let $G^{\mathbb C}$ be a complex semisimple Lie group.
Let $P$ be a parabolic subgroup in $G^{\mathbb C}$ and 
$G$ a non-compact real form of $G^{\mathbb C}$. 
Then an open $G$-orbit in the flag manifold $G^\CC/P$ 
is called a {\it flag domain} (cf. \cite{Huck_Fels_Wolf, Wolf}).
Without loss of generality, assume that the $G$-orbit of $eP$
is open in $G^\CC/P$.
Then we call a flag domain $F=G\cdot z_0 = G/V$ 
with $z_0=eP\in G^{\mathbb C}/P$ a {\it canonical flag domain}
if $V:=G\cap P$ is a compact Lie subgroup
containing a maximal torus of $G$. 
Two of the most important examples of canonical flag domains
are the period domains and the Mumford--Tate domains.
In \cite{Burstall_Rawnsley}, it is explained that 
the canonical flag domain $F$ admits a HBGD, denoted by $H_F$, 
which is referred to as the  
{\it superhorizontal distribution}.
We are able to show that the canonical flag domain $F$ with $H_F$ is complete Kobayashi hyperbolic 
by exploiting that it has holomorphic sectional curvature 
bounded from above by a negative constant. 
For more details, see Section 3 and 4.
Note that flag domains that are not Hermitian symmetric spaces of non-compact type,
always contain compact complex homogeneous submanifolds and hence 
those are not Kobayashi hyperbolic.

In \cite{Nakajima}, Nakajima proved that 
if a complex manifold $M$ is Kobayashi hyperbolic, 
then it is biholomorphic to a Siegel domain of second type, i.e.,
to a bounded homogeneous domain in a Euclidean space.
As an analogue of Nakajima's theorem, we can characterize canonical flag domains 
through Kobayashi hyperbolicity with HBGDs.
We say that $(M,D)$ is {\it G-homogeneous complex manifold with an invariant distribution} 
if $G$ acts transitively, holomorphically and almost effectively 
on $M$ and $D$ is invariant with respect to the action of $G$.
\begin{thm}\label{main theorem}
Let $G$ be a linear semisimple Lie group.
Then $(M,D)$ is a Kobayashi hyperbolic $G$-homogeneous complex manifold 
with an invariant holomorphic bracket generating distribution
if and only if the universal covering of $M$ is a canonical flag domain 
and the induced distribution is the superhorizontal one.
\end{thm}
The proof of Theorem \ref{main theorem} 
will be given in Section \ref{homogeneous manifolds}.

The relation between Kobayashi hyperbolicity 
and negative curvature is induced by a generalization of Schwarz's lemma: 
let $f\colon \Delta \rightarrow M$ be a holomorphic mapping,
where $M$ is a complex manifold equipped with a Hermitian metric $g$. 
Suppose that the holomorphic sectional curvature of $M$
is bounded from above by a negative constant. 
Then $f^*g \leq cg_\Delta$, for some positive constant $c$
and the Poincar\'e  metric $g_\Delta$ on $\Delta$.
A generalization of this Schwarz's lemma for holomorphic discs 
tangential to a holomorphic distribution with holomorphic sectional curvature
bounded from above by a negative constant is already well-known 
(for instance see \cite{Car-Mul-Pet}).
In this paper we will consider generalized Schwarz's lemma of Yau given in \cite{Yau}
to obtain the version given in \cite{Car-Mul-Pet}.
Even though we only need Schwarz's lemma given
in Corollary \ref{Alfors-Schwarz lemma for holomorphic distribution}, 
generalized Schwarz's lemma of Yau 
for complex manifolds with holomorphic distributions
will be described for future reference.

Here is the outline of the paper.
In section \ref{definition and properties}, we recall notions related to holomorphic vector fields
and prove Theorem \ref{finiteness}. 
The properties are given without proof  
since they can be obtained through the same proofs used in \cite{Kobayashi}
and in the papers cited.
In section \ref{Schwarz lemma}, we explain the relation between the 
Kobayashi hyperbolicity of $(M,D)$ and the holomorphic sectional curvature
bounded from above by a negative constant.
In section \ref{flag domains}, we describe two examples of 
Kobayashi hyperbolic complex manifolds with HBGDs.
In section \ref{homogeneous}, we prove Theorem \ref{main theorem}.

Throughout this paper, the Roman letters $a,b,c,\ldots$ run from $1$ to $m$, 
the Roman letters $i, j, k,\ldots$ run from $1$ to $d$, 
the Greek letters $\al, \be, \ga, \ldots$ run from $d+1$ to $m$ 
and the Greek letters $\sigma, \mu, \nu, \ldots$ run from $1$ to $n$. 
Let $\mathfrak g,\, \mathfrak v, \, \mathfrak k,\ldots$ denote 
Lie algebras and 
$G,\, V,\, K,\dots$ the Lie groups corresponding to 
$\mathfrak g,\, \mathfrak v, \, \mathfrak k,\ldots$.
Given a subspace $\mathfrak{a}\subset \mathfrak{g}$, 
let $\mathfrak{a}^\mathbb C$ denote its complexification.

\bigskip
{\bf Acknowledgement}
The author would like to thank Jun-muk Hwang for sharing his idea on the subject.
Thanks are also due to Jae-Hyun Hong, Alan Huckleberry and Sung-Yeon Kim
for useful discussions on this paper.
{The author would like to thank anonymous referee for the valuable comments.}
This research was partially supported by “Overseas Research Program for Young Scientists” 
through Korea Institute for Advanced Study (KIAS).

\section{A Kobayashi pseudo-distance for complex manifolds with  holomorphic bracket generating distributions}
\label{definition and properties}
\subsection{Holomorphic vector fields}\label{Holomorphic vector fields}
Let $M$ be a complex manifold 
and $T_M$ its holomorphic tangent bundle.
A holomorphic vector field $X$ on $M$ 
is a holomorphic section of $T_M$. 
Since $T_M$ is naturally isomorphic to the real tangent bundle $TM$, 
we can identify $X$ with a real vector field 
that we continue to denote by $X$. 
A flow $\phi_X$ associated to $X$, 
defined on an open subset $U$ of $\RR \times M$ 
containing $\{0\}\times M$, is given in the following way: 
for $(t,p) \in U$, we set $\phi_X(t,p) := c_p(t)$, 
where $c_p \colon (−a(p),b(p)) \rightarrow M$ 
is the unique maximal solution to the initial value problem
\begin{equation}
\begin{aligned}
\frac{d}{dt}c(t)&=X\circ c_p(t),\\
  c_p(0)&=p. 
  \end{aligned}
\end{equation}
Let $\phi_X(t)p$ denote $\phi_X(t,p)$.
Fix $p_0\in M$ and a sufficiently small $\epsilon>0$.
Then, there is an open neighborhood $\widetilde U$ of $p_0$
such that $p\mapsto \phi_X(t)p$ is holomorphic on $\widetilde U$
and  for each $t\in (-\epsilon, \epsilon)$, the mapping $\phi_X(t)$ is a holomorphic diffeomorphism from $\widetilde U$ onto
its image.
Define the complex flow of $X$ to be
\begin{equation}
\Phi_X(s+it) := \phi_X(s)\circ \phi_{iX}(t).
\end{equation}
Then, for sufficiently small $\epsilon'>0$, the complex flow $\Phi_X$ is a well-defined holomorphic mapping on $\{z\in \mathbb C: |z|< \epsilon'\}$ 
and satisfies $\frac{d}{dz} \Phi_X(z)p = X\circ \Phi_X(z)p$.

\subsection{Piecewise connected horizontal discs}
Let $M$ be a complex manifold of dimension $m$ 
and $D$ a HBGD of rank $d$ on $M$.
Fix $x\in M$. 
Let $X_1,\ldots, X_d$ be holomorphic vector fields 
on a small neighborhood $U$ of $x$ 
such that  $\{ X_1|_y,\ldots, X_d|_y\}$ is a basis of $D_y$ for all $y\in U$.
Choose suitable $X_{i_{j,k}}$ such that 
\begin{eqnarray*}
X_{d+1} &=& [X_{i_{1,1}},X_{i_{1,2}}], \,\,
X_{d+2} = [X_{i_{2,1}},X_{i_{2,2}}], \ldots \\
&&\ldots, X_m = [X_{i_{{m-d,1}}}, \ldots, [X_{i_{m-d,k-1}},X_{i_{{m-d,k}}}], \ldots]
\end{eqnarray*} 
form a basis of $T_yM$ together with $X_1,\ldots,X_d$ at all points $y\in U$.
Then there is a small open neighborhood  $\widetilde U\subset U$ of $x$ 
and $W\subset \CC^m$ such that the holomorphic mapping
$
F\colon W \rightarrow \widetilde{U}
$ defined by 
\begin{equation}
F(t_1,\dots,t_m) = \phi_1(t_1)\circ\dots \circ\phi_m(t_m)x
\end{equation} is a holomorphic diffeomorphism.
Here $\phi_1,\ldots, \phi_m$ are the complex flows associated to
the vector fields $X_1,\ldots,X_m$ described in Section \ref{Holomorphic vector fields}.
In particular, for each $j$ with $d+1 \leq j \leq m$, the complex flow 
$\phi_j$ is a finite composition of $\phi_i$ and $\phi_i^{-1}$ with $1 \leq i \leq d$.
Hence, for fixed $t_1,\ldots, t_m$, there are holomorphic discs connecting $x$ and $\phi_1(t_1)\circ \cdots\circ \phi_m(t_m)x$.
Note that $$ \frac{dF}{dt_i}(t_1,\dots, t_m) = d\phi_1(t_1)\circ \dots \circ d\phi_{i-1}(t_{i-1})X_i|_{\phi_{i+1}(t_{i+1})\circ \dots \circ\phi_m(t_m)x} $$
and in particular $dF|_0 = id$. 
This implies that there is an open neighborhood of $x$ such that any two points can be connected 
by finite number of holomorphic discs tangential to $D$.
Using standard methods, it is possible to show that local connectivity implies global connectivity.
So, if $M$ is connected, 
every two points in $M$ can be connected by a chain of holomorphic discs that are tangential to $D$.
That is for any $x, \, y \in M$, we can take a finite set $\{x_0,\ldots, x_N\}\subset M$ 
and a set $\{a_1,\ldots,a_N\}\subset \Delta$ with $x=x_0$ and $y=x_N$ such that there are holomorphic discs $f_j\colon \Delta \rightarrow M $ tangential to $D$ for each $j=1,\ldots, N$ with $f_j(0)=x_{j-1}$ and $f_j(a_j)=x_j$.
Hence this completes the proof of Theorem \ref{finiteness}.

\subsection{Properties of the Kobayashi pseudo-distance and the Kobayashi pseudo-metrics for ${\bf (M,D)}$}
In this section, selected properties of the Kobayashi pseudo-distance given in Definition \ref{definition of the Kobayashi distance} will be presented without proof.
Here $M$, $M_1$, and $M_2$ are complex manifolds equipped with HBGDs $D$, $D_1$ and $D_2$, respectively.

\begin{prop}\label{Kobayashi properties}
$\hbox{}$
\begin{enumerate}
\item The Kobayashi pseudo-distance  satisfies the triangle inequality.
\item Let $f$ be a holomorphic mapping from $(M_1,D_1)$ to $(M_2, D_2)$ such that $df(D_1)\subset D_2$. Then for any $x,\,y \in M_1$ and $v\in D_1$ one has
\begin{equation}
\begin{aligned}
d_{M_2,D_2}(f(x), f(y)) &\leq d_{M_1,D_1}(x,y),\\
k_{M_2,D_2}(df(v))&\leq k_{M_1,D_1}(v).
\end{aligned}
\end{equation}
\end{enumerate}
\end{prop}

\begin{prop}\textup{(Demailly} \cite[ Proposition 1.5.]{Demailly}\textup{)}
$k_{M,D}$ is upper semicontinuous on the total space of $D$. If $M$ is compact,
$(M,D)$ is infinitesimally Kobayashi hyperbolic if and only if there are no non-constant entire curves
$f \colon \mathbb C \rightarrow M$ tangential to $D$. In that case, $k_{M,D}$ is a continuous (and positive definite) Finsler metric on $D$.
\end{prop}

\begin{prop}[cf. \cite{Royden, Abate}]
For $x,y \in M$ we have
\begin{equation} 
d_{M,D}(x,y) = \inf_\gamma \int k_{M,D}(\gamma'(t))dt,
\end{equation}
where the infimum is taken over all piecewise differentiable curve $\gamma \colon [0,1]\rightarrow M$ such that $\gamma'(t)\in D$ for all $t\in [0,1]$ and $\gamma(0)=x,\, \gamma(1)=y$.
\end{prop}

\begin{cor}
\hfill
\begin{enumerate}
\item
$(M,D)$ is Kobayashi hyperbolic if and only if $(M,D)$ is infinitesimally Kobayashi hyperbolic.
\item
The Kobayashi distance $d_{M,D}$ is continuous.
\end{enumerate}
\end{cor}

Suppose that $(M,D)$ is Kobayashi hyperbolic.
Then the following theorem of van Dantzig and van
der Waerden in \cite{Dantzig_Waerden_1928} applies:
{\it let $X$ be a connected locally compact metric space.
Then the isometry group of $X$ is locally compact in the compact-open topology
and the isotropy subgroup at any point of $X$ is compact.
Furthermore if $X$ is compact, then the isometry group of $X$ is also compact.}
Hence the isometry group of $M$ with respect to $d_{M,D}$ is locally compact.
By a theorem of Bochner and Montgomery in \cite[Theorem A]{Bochner_Montgomery_1945}
we obtain the following:
\begin{prop}\label{compact isotropy}
Suppose that $(M,D)$ is Kobayashi hyperbolic.
Then the set of holomorphic diffeomorphisms of $M$ preserving $D$, say $\text{Aut}(M,D)$, 
is a Lie group and at any point in $M$, 
the isotropy subgroup of $\text{Aut}(M,D)$ is compact.
\end{prop}

\begin{cor}\label{No complex action}
Suppose that $(M,D)$ is a Kobayashi hyperbolic $G$-homogeneous complex manifold 
with an invariant distribution $D$. Then $G$ is not a complex Lie group.
\end{cor}

\begin{prop} 
Let $\widetilde{M}$ be a covering manifold of $M$ with the covering projection $\pi\colon \widetilde{M}\rightarrow M$. 
Define an induced HBGD, denoted by $\widetilde{D}$, on $\widetilde{M}$ by $d\pi^{-1}(D)$.
\begin{enumerate}
\item Let $x,\,y\in M$ and choose $\widetilde{x}\in \widetilde{M}$ such that $\pi(\widetilde{x})=x$. Then 
\begin{equation}
d_{M,D}(x,y) = \inf_{\widetilde{y}} d_{\widetilde{M}, \widetilde{D}}(\widetilde{x},\widetilde{y})
\end{equation}
where the infimum is taken over all $\widetilde{y}\in \widetilde{M}$ such that $y=\pi(\widetilde{y})$.
\item
$(M,D)$ is (complete) Kobayashi hyperbolic if and only if $(\widetilde{M}, \widetilde{D})$ is (complete) Kobayashi hyperbolic.
\end{enumerate}
\end{prop}

\begin{prop}\label{fiber bundle}
Let $\pi\colon P\rightarrow M$ be a holomorphic fiber bundle over $M$.
Suppose that $(M,D)$ is (complete) Kobayashi hyperbolic. 
Let $T_P=\mathcal{H}\oplus \mathcal{V}$ be a decomposition
where $\mathcal{V}$ is the vertical distribution and 
$\mathcal{H}$ is any horizontal distribution (i.e., $d\pi|_{\mathcal H}\colon \mathcal H\rightarrow T_M$ is an isomorphism).
Set $\widetilde D = (d\pi|_{\mathcal{H}})^{-1}(D) \subset \mathcal{H}$.
\begin{enumerate}
\item
If $\widetilde D$ is bracket generating in $T_P$, 
then $(P,\widetilde D)$ is (complete) Kobayashi hyperbolic.
\item
If the fiber of $P$ is (complete) Kobayashi hyperbolic, then $(P,\widetilde D\oplus \mathcal{V})$ is also (complete) Kobayashi hyperbolic.
\end{enumerate}
\end{prop}

\begin{prop}\label{embedding}
Suppose that there is an embedding $\iota \colon M_1\rightarrow M_2$ such that $d\iota (D_1)\subset D_2$. 
If $(M_2,D_2)$ is (complete) hyperbolic, then $(M_1,D_1)$ is also (complete) hyperbolic.
\end{prop}

\begin{prop}
Suppose that $(M,D)$ is complete Kobayashi hyperbolic. 
Then the set of holomorphic mappings from $M_1$ to $(M,D)$ 
tangential to $D$ is a normal family: every sequence of holomorphic maps 
from $M_1$ to $M$ tangential to $D$ either has a subsequence 
that converges uniformly on compact subsets 
or has a compactly divergent subsequence.
\end{prop}

\section{Kobayashi hyperbolicity and negative curvature}\label{Schwarz lemma}

\subsection{ Generalized Schwarz's lemma for $\mathbf{(M,D)}$ 
and Kobayashi hyperbolicity for $\mathbf{(M,D)}$ with negative holomorphic sectional curvatures}
Let $M$ be a complex manifold of complex dimension $m$ 
and $D$ a holomorphic distribution on $M$ of rank $d$ with a Hermitian metric $g$. 
Let $A^k(D)$ denote the set of $D$-valued $k$-forms on $M$.
Given a Hermitian metric $g$, there is a unique connection, called the Chern connection, 
$\nabla\colon A^0(D)\rightarrow A^1(D)$ which is compatible with 
both the complex structure of $M$ 
and the metric.
Choose a local orthonormal frame $e_1,\dots,e_d$ of $D$ and denote its dual frame by $\theta_1,\dots, \theta_d$, i.e., $g=\sum\theta_i\overline\theta_i$.
Denote $\theta_{ij}$ a connection 1-form, i.e.,
$$
\nabla e_i = \sum \theta_{ij} \otimes e_j.
$$ 
If one extends $g$ to a Hermitian metric of $T_M$ and decompose it as $T_M = D \oplus D^\perp$ where $D^\perp$ is the orthogonal complement of $D$ in $T_M$, then $\nabla = \pi_D\circ \nabla_M$ where $\nabla_M$ is the Chern connection with respect to the extended Hermitian metric on $T_M$ 
and $\pi_D$ is the orthogonal projection onto $D$.
By adding a local frame of $D^\perp$ we can complete the given frame to $e_1, \dots, e_m$ 
and the dual frame to $\theta_1,\dots, \theta_m$. 
Then $d\theta_i - \theta_{ij}\wedge \theta_j$ can be expressed by 
\begin{equation}
d\theta_i = \theta_{ij}\wedge \theta_j + \Theta_i + N_i,
\end{equation}
where $N_i$ contains $\theta_\alpha$ or $\overline\theta_\alpha$ for $d+1\leq  \alpha \leq m$ and $\Theta_i$ consists of $\theta_j$ and $\overline\theta_j$ for $1\leq j \leq d$.
We can write
\begin{equation}
\Theta_i = \Theta_i^{2,0} + \Theta_i^{1,1} + \Theta_i^{0,2}
\end{equation}
with an $(a,b)$-form $\Theta_i^{a,b}$. Note that
$\Theta_i^{0,2}=0$ and $\Theta_i^{1,1}=0$ 
since $M$ is a complex manifold,
and $\nabla$ is the Chern connection.
The curvature tensor is given by
\begin{equation}
\Theta_{ij} = d\theta_{ij} - \theta_{ik}\wedge\theta_{kj}
= R_{ija\overline{b}}\theta_a\wedge\overline\theta_b.
\end{equation}
For $X = \sum X_i e_i$ and $Y=\sum Y_i e_i$ in $D_p$, 
define the holomorphic bisectional curvature of $(M, D, g)$ in the directions $X$ and $Y$ by
\begin{equation}
\text{Bisec}_{(D,g)}(X,Y)= \frac{\sum R_{ijk\overline{l}} X_i\overline{X}_jY_k\overline{Y}_l}{\sum |X_i|^2 \sum |Y_i|^2},
\end{equation}
and the holomorphic sectional curvature of $(M,D,g)$ in the direction $X$ by
\begin{equation}
H_{(D,g)}(X) = \text{Bisec}_{(D,g)}(X,X).
\end{equation}

For a complex manifold $N$ equipped with a Hermitian metric $h$,
let $\{\omega_\sigma\}$ be an orthonormal frame of $(N,h)$.
For a holomorphic mapping
$f\colon N\rightarrow M$ tangential to $D$,
let $u=\sum_{i\sigma} a_{i\sigma}\overline a_{i\sigma}$, 
where 
$f^*(\theta_i) = \sum_\sigma a_{i\sigma}\omega_\sigma$.
Then by the Omori-Yau maximum principle 
and the same calculations of the Laplacian of $u$ given in \cite{Yau}, 
we have the following theorem.

\begin{thm} \label{Schwarz lemma for holomorphic distribution}
 Let $(N,h)$ be a complete Hermitian manifold with 
 Ricci curvature bounded from below by a constant $K_1$.
 Let $(M,D,g)$ be a complex manifold  with a holomorphic distribution $D$ equipped 
with a Hermitian metric $g$. 
Suppose that the holomorphic bisectional curvature of $(M, D,g)$ is bounded from above by a negative constant $K_2$. Then, for any holomorphic mapping $f$ from $N$ into $M$ tangential to $D$, we have
\begin{equation}
f^*(g)\leq \frac{K_1}{K_2}\, h.
\end{equation}
\end{thm}

\begin{cor} [cf. \cite{Car-Mul-Pet}]\label{Alfors-Schwarz lemma for holomorphic distribution}
 Let $(\Delta,g_\Delta)$ be the unit disc in $\mathbb C$ endowed with the Poincar\'e metric $g_\Delta$ with curvature $-1$.
Let $(M,D,g)$ be a complex manifold  with a holomorphic distribution $D$ and $g$ a Hermitian metric on $D$. Suppose that the holomorphic sectional curvature of $(M,D,g)$ is bounded from above by a negative constant $K$. Then, we have
\begin{equation}
f^*(g)\leq \frac{1}{K} \, g_\Delta
\end{equation}
for any holomorphic mapping $f$ from $\Delta$ into $M$ tangential to $ D$.
\end{cor}

For any $x,\,y\in M$, define the {\it Carnot-Caratheodory} distance with respect to $(M,D,g)$ as
\begin{equation}
\rho_{M,D,g}(x,y) = \inf_\gamma  \int g(\gamma'(t), \gamma'(t))^{1/2}dt
\end{equation}
where the infimum is taken over all $\gamma\colon [0,1]\rightarrow M$ tangential to $D$ such that $\gamma(0)=x,\, \gamma(1)=y$. If there is no curve satisfying this condition, set $\rho_{M,D,g}(x,y)=\infty$.

\begin{thm}[cf. \cite{Montgomery}]\label{Carnot-Caratheodory distance}
Suppose that $D$ is bracket generating. Then
$\rho_{M,D,g}$ is finite, continuous on $M$ and induces the manifold topology.
\end{thm}
Let $\rho_{N,h}$ be a distance function of $h$ on $N$.
Then Theorem \ref{Schwarz lemma for holomorphic distribution} implies that for $a,b\in N$
one has
\begin{equation}
\rho_{M,D,g}(f(a),f(b))\leq \sqrt{\frac{K_1}{K_2}}\rho_{N,h}(a,b).
\end{equation}

If the holomorphic sectional curvature of $(M,D,g)$ is bounded from above by a negative constant, 
for any holomorphic mapping $f\colon \Delta \rightarrow M$ tangential to $ D$, 
then by Corollary \ref{Alfors-Schwarz lemma for holomorphic distribution} one has
$f^* g \leq c g_\Delta$ for some positive constant $c$. 
This implies that for any $a,\,b\in \Delta$ 
$$
\rho_{M,D,g}(f(a),f(b))\leq \sqrt{c} \rho_\Delta(a,b).
$$
Suppose that $d_{M,D}(x,y)=0$ for some $x,\,y\in M$. 
This implies that there are sequences of elements $a_i,\, b_i \in \Delta$ and 
holomorphic mappings $f_i\colon \Delta\rightarrow M$ such that $f(a_i)=x,\, f(b_i)=y$ and $\rho_\Delta(a_i,b_i) \rightarrow 0$ as $i\rightarrow \infty$. Hence $\rho_{M,D,g}(x,y)=0$ and as a consequence we have $x=y$ by Theorem \ref{Carnot-Caratheodory distance}. 
Thus we obtain the following:
\begin{cor}\label{negative curvature}
Suppose that the holomorphic sectional curvature of 
$(M,D,g)$ is bounded from above by a negative constant. 
Then $(M,D)$ is Kobayashi hyperbolic. 
If the Carnot-Caratheodory distance of $g$ is complete, 
then the Kobayashi distance of $(M,D)$ is also complete.
\end{cor}

\begin{remark} \label{non-holomorphic-bracket-generating}
Suppose that $D$ is a subbundle of the holomorphic tangent bundle of $M$. 
Here $D$ does not need to be bracket generating.
Let $g$ be a Hermitian metric on the tangent bundle of $M$ and 
$\Theta$ denote the Chern curvature tensor of $g$ on $M$. 
If $\frac{g(\Theta(\zeta, \xi)\zeta, \xi)}{g(\zeta, \zeta)g(\xi,\xi)}$ 
is bounded from above by a negative constant for all $\zeta, \, \xi\in D$,
then the same argument in Theorem \ref{Schwarz lemma for holomorphic distribution} applies.
That is, for every holomorphic mapping $f \colon N\rightarrow M$ tangential to $D$, 
the inequality $f^*(g)\leq \frac{K_1}{K_2} \, h$ holds.
Furthermore Corollary \ref{Alfors-Schwarz lemma for holomorphic distribution} 
also holds as \cite[13.4.1 Schwarz's lemma]{Car-Mul-Pet}.
\end{remark}

\section{Examples of Kobayashi hyperbolic manifold $(M,D)$}\label{flag domains}
\subsection{Canonical flag domains with superhorizontal distributions}


In this section, we only consider canonical flag domains.
Let $T$ denote a maximal torus, 
$V=G\cap P$ the isotropy group of $G$ at $z_0$ 
and $K$ the unique maximal compact subgroup of $G$ containing $V$.
Let $\mathfrak{t}\subset \mathfrak{v}\subset \mathfrak{k} \subset \mathfrak{g}$
be the Lie algebras of $T\subset V\subset K\subset G$, respectively.
Let $\mathfrak{g} = \mathfrak{k}\oplus \mathfrak{q}$ 
be the Cartan decomposition. 
The Killing form $B$ of $\mathfrak g$ is negative definite on $\mathfrak{k}$ 
and positive definite on  $\mathfrak{q}$. 
Let $ \mathfrak{h} := \mathfrak{t}^\mathbb C$ be
a Cartan subalgebra of $\mathfrak{g}^{\mathbb C}$.
Let $\Phi = \Phi(\mathfrak{g}^\mathbb C, \mathfrak{h})$ 
denote the set of roots of $\mathfrak{g}^\mathbb C$ with respect to $\mathfrak h$.
Given a root $\alpha\in \Phi$, 
let $\mathfrak{g}^\alpha\subset \mathfrak{g}^\mathbb C$ 
denote the associated root space. 
Given a subspace $\mathfrak{s}\in \mathfrak{g}^{\mathbb C}$, 
set
$$
\Phi(\mathfrak{s}) := \{\alpha\in \Phi \colon \mathfrak{g}^\alpha \subset \mathfrak{s}\}.
$$
We say that a root $\alpha$ is compact (non-compact) 
when $\alpha \in \Phi(\frak k^{\mathbb C})$ ($\alpha \in \Phi(\frak q^{\mathbb C})$).
Fix a Borel subalgebra $\frak b$ such that 
$\mathfrak{h}\subset \mathfrak{b}\subset \mathfrak{g}^\mathbb C$. 
Then 
$$
\Phi^+ = \Phi(\mathfrak{b})
$$
determines a set of positive roots. 
Let $\Phi^-$ denote the set of corresponding negative roots. 
Given a subspace $\mathfrak{s}\subset\mathfrak{g}^\mathbb C$, define 
$$
\Phi^+(\mathfrak{s}) := \Phi(\mathfrak{s})\cap \Phi^+,\quad 
\Phi^-(\mathfrak{s}) := \Phi(\mathfrak{s})\cap \Phi^-.
$$
Let $\{\alpha_1,\ldots, \alpha_r\}\subset \Phi^+ $ denote the corresponding simple roots. 
For each $\alpha\in\Phi$, one can choose vectors $e_\alpha\in \mathfrak{g}^\alpha$ and $h_\alpha \in \mathfrak h$ such that
\begin{enumerate}
\item $B(e_\alpha, e_\beta) = \delta_{\alpha, -\beta}\,\, $  and $\quad [e_\alpha, e_{-\alpha}] = h_\alpha$;
\item $B(h_\alpha, x) = \alpha(x)$ for every $x\in \mathfrak{h}$;
\item $[e_\alpha, e_\beta]=0$, if $\alpha \neq -\beta$ and $\alpha+\beta \notin \Phi$;
\item $[e_\alpha, e_\beta] = N_{\alpha,\beta}\,e_{\alpha + \beta}$, 
if $\alpha, \beta, \alpha+\beta \in \Phi$ 
where $N_{\alpha,\beta}$ are nonzero real constants such that 
$$
N_{-\alpha, -\beta} = -N_{\alpha,\beta},\quad
N_{-\alpha,-\beta} = N_{-\beta, \alpha+\beta} = N_{\alpha + \beta, -\alpha};
$$
\item 
for the complex conjugate $\sigma$ of $\mathfrak{g}$ in $\mathfrak{g}^\mathbb C$,
we have
$
\sigma(e_\alpha) = \epsilon_\alpha e_{-\alpha}
$ 
where $\epsilon_\alpha = -1$ if $\alpha$ is compact and $\epsilon_\alpha = 1$ if $\alpha$ is non-compact. \label{conjugation of D}
\end{enumerate}

Let $\omega^\alpha$ be dual covectors of $e_\alpha\in \mathfrak{g}^\alpha$
in $(\mathfrak{g}^{\mathbb C})^*$.
Let $\{T^1,\ldots, T^r\}$ be the basis of $\mathfrak{h}$ 
dual to the simple roots $\{\alpha_1,\ldots, \alpha_r\}$. 
Set
$
T := \sum_{\alpha\in\Phi \setminus \Phi(\mathfrak{v}^\mathbb C)} T^i$.
Then $\mathfrak g^{\mathbb C}$ can be decomposed as
\begin{equation}
\mathfrak{g}^\mathbb C = \mathfrak{g}_k\oplus \cdots \oplus \mathfrak{g}_1\oplus \mathfrak{g}_0\oplus \mathfrak{g}_{-1}\oplus\cdots \oplus \mathfrak{g}_{-k},
\end{equation}
where $\mathfrak{g}_l = \{\zeta\in \mathfrak{g}^\mathbb C : [T, \zeta] = l\zeta\}$ with $\l\in \ZZ$. Notice that $[\frak g_i, \frak g_j]\subset \frak g_{i+j}$ and 
\begin{equation}\label{odd even}
\mathfrak{g}_0 = \mathfrak{v}^\mathbb C, \
\mathfrak k^{\mathbb C} = \sum_{i \text{ even}} \mathfrak g_i \ \text{ and } \
\mathfrak q^{\mathbb C} = \sum_{i \text{ odd }}\mathfrak g_i. 
\end{equation}
Denote 
$\mathfrak{g}_+ = \bigoplus_{j> 0} \mathfrak{g}_j$ and 
$\mathfrak{g}_- = \bigoplus_{j<0} \mathfrak{g}_j$.
Then a homogeneous complex structure on $F$ is given by specifying 
$T^\CC F = TF\otimes \CC = T^{1,0}F \oplus T^{0,1}F$ with $
T^{1,0}F = G\times_V \mathfrak{g}_-$ and $T^{0,1}F = G\times_V \mathfrak{g}_+,
$
where $T^{1,0}F$ is the holomorphic tangent bundle of $F$.
Let \begin{equation}
H_F:= G\times_V \mathfrak{g}_{-1}
\end{equation}
 denote  the {\it superhorizontal distribution} given in \cite[Chapter 4]{Burstall_Rawnsley}.
$H_F$ is the unique HBGD contained in $G\times_V\frak q$.
 Note that if $F$ is a Hermitian symmetric space of non-compact type, 
 then $H_F$ is the holomorphic tangent bundle $T_F$.
Let $g$ be the G-invariant Hermitian metric on $H_F$ defined by 
\begin{equation}\label{metric on flag domain}
g(\zeta, \xi) = B(\zeta,\sigma(\xi)).
\end{equation}
By the expression of Chern curvature given in \cite{Griffiths-Schmid},
we obtain
\begin{equation}\label{curvature}
\begin{aligned}
\Theta_F = -\sum_{\alpha, \beta\in \Phi^+ \setminus \Phi^+({\mathfrak{k}^\mathbb C})}& \text{ad}_{[ e_\alpha, e_{-\beta}]_{\mathfrak{v}^\mathbb C}} \otimes \omega^\alpha \wedge \overline{\omega}^\beta \\
  &
+\sum_{\alpha, \beta\in \Phi^+(\mathfrak{k}^\mathbb C)\setminus\Phi^+(\mathfrak{v}^\mathbb C)} \text{ad}_{[ e_\alpha, e_{-\beta}]_{\mathfrak{v}^\mathbb C}} \otimes \omega^\alpha \wedge \overline{\omega}^\beta,
\end{aligned}
\end{equation}
where $[e_{\alpha}, e_{-\beta}]_{\mathfrak v^{\mathbb C}}$ denotes the projection of $[e_\alpha, e_{-\beta}]$ onto $\mathfrak v^{\mathbb C}$. 
For $\zeta, \xi \in H_F|_{eV}$, the holomorphic bisectional curvature of $(F, H_F,g)$ is given by
\begin{equation}
\text{Bisec}_{F,H_F}(\zeta,\xi) 
=- \frac{B([[\zeta,\sigma(\zeta)],\xi],\sigma(\xi))}{g(\zeta,\zeta) g(\xi,\xi)}
=-\frac{B([\zeta,\sigma(\zeta)],[\xi, \sigma(\xi)])}{g(\zeta,\zeta) g(\xi,\xi)}.
\end{equation}

Let $\zeta\in \mathfrak g_{-1}$. 
Then 
$
[\zeta, \sigma(\zeta)] 
\in \sqrt{-1}\mathfrak{v}$. 
Note that $B$ is positive definite on $\sqrt{-1}\mathfrak{v}$.
Hence, there should be a positive constant $c$ such that 
\begin{equation}
H_{F,H_F}(\zeta)
=-\frac{B([\zeta,\sigma(\zeta)],[\zeta, \sigma(\zeta)])}{g(\zeta,\zeta)^2}
<-c<0.
\end{equation}
In view of Corollary \ref{negative curvature}, of the homogeneity of $F$ 
and of the ball-box theorem given in \cite[Theorem 2.4.2]{Montgomery}, we obtain the following.
\begin{thm}\label{Flag domain}
 Canonical flag domains with superhorizontal distributions are complete Kobayashi hyperbolic.
\end{thm}

\begin{cor}
Let $\Gamma$ be an uniform lattice of $G$, 
i.e., $\Gamma\setminus G$ is compact.
Then $\Gamma\setminus F$, with the distribution induced by
 the superhorizontal distribution of $F$, is compact  and complete Kobayashi hyperbolic.
\end{cor}

Suppose that $F$ is not a Hermitian symmetric space of non-compact type. 
In \cite[Lemma 5.1]{Wolf}, Wolf showed that there exists the unique closed $K^{\mathbb C}$-orbit in $F$.
Since it is a compact complex homogeneous submanifold of $F$,
this implies that $F$ is not Kobayashi hyperbolic.
This $K^{\mathbb C}$-orbit is called the {\it base cycle} of $F$.

\begin{cor}
There is no holomorphic embedding from a flag domain 
which is not Hermitian symmetric into the canonical flag domain 
tangential to the superhorizontal distribution.
\end{cor}
\begin{proof}
Let $F_1,\, F_2$ be flag domains where $F_1$ is not a Hermitian symmetric space,
and $D_2$ is the superhorizontal distribution of $F_2$.
Let $\iota\colon F_1\rightarrow F_2$ be a holomorphic embedding tangential to $D_2$. 
Then by Proposition \ref{embedding}, one has $d_{F_2, D_2}(\iota(x),\iota(y))\leq d_{F_1}(x,y)$.
Let $x\neq y$ in the base cycle of $F_1$. Then $d_{F_1}(x,y)=0$ 
and this induces the contradiction to that $(F_2, D_2)$ is Kobayashi hyperbolic.
\end{proof}
\begin{remark}
In \cite{Car-Mul-Pet}, the authors proved that the Chern connection on $T_F$ 
with respect to the same Hermitian metric as in \eqref{metric on flag domain} has 
holomorphic sectional curvature bounded from above 
by a negative constant in non-compact direction 
$\mathfrak q^{\mathbb C}$.
Let $\iota\colon F_1\rightarrow F_2$ be a holomorphic embedding tangential to $D_2$.
Then, by Remark \ref{non-holomorphic-bracket-generating} we can obtain that 
$\rho_{F_2, T_{F_2}, g}(\iota(x), \iota(y)) \leq \rho_{F_2,D_2,g}(\iota(x), \iota(y))
\leq c d_{F_1}(x,y)$ for a positive constant $c$.
Hence, there is no holomorphic embedding from a flag domain which
is not Hermitian symmetric into the canonical flag domain tangential to the $\mathfrak q^{\mathbb C}$-direction.
\end{remark}

\subsection{Kobayashi hyperbolicity of the complex Euclidean spaces}
In \cite{Forstneric_2016}, Forstneric constructed a holomorphic contact distribution $D$ 
on $\mathbb C^{2n+1}$ such that $(\mathbb C^{2n+1}, D)$ is Kobayashi hyperbolic 
for any $n\in \mathbb N$.
We generalize his construction to some HBGD on complex Euclidean spaces.

Let $D$ be a HBGD on $\mathbb C^n$ of rank $d$.
Since $\mathbb C^n$ is a contractible Stein manifold,
by the Grauert-Oka principle, both $D$ and $T_{\mathbb C^n}/D$ are holomorphically trivial vector bundles.
Let $X_1,\ldots, X_d, X_{d+1}, \ldots, X_n$ be a global frame of $\mathbb C^n$
such that $X_1,\ldots, X_d$ is a frame of $D$.
Let $X_j = \sum_{i=1}^n X_{ji} \frac{\partial }{\partial z_i}$ for each $j=1,\ldots, n$
with holomorphic functions $X_{ji}$ on $\mathbb C^n$.
Denote the matrix $(X_{ji})_{j,i=1,\ldots, n}$ by $X$.
\begin{eqnarray}\label{invertible assumption}
&\text{ Suppose that there is a } d\times d \text{ submatrix in } (X_{ji})_{j=1,\ldots d, i=1,\ldots, n}\\\nonumber
&\quad\quad\text{ which is invertible at every point in } \mathbb C^n.
\end{eqnarray}

Then there is a Fatou-Bieberbach map $\Phi\colon \mathbb C^n\rightarrow \mathbb C^n$
such that $\left( \mathbb C^n, d\Phi^{-1}(D|_{\Phi(\mathbb C^n)}) \right)$ is Kobayashi hyperbolic.
This can be achieved through Lemma \ref{lem1} and Lemma \ref{lem2}.
Those proof are analogue to the corresponding lemmas of Forstneric.
Note that the condition \eqref{invertible assumption} is not true in general.

\begin{lem}\label{lem1}\textup{(Generalization of }\cite[Lemma 2.1]{Forstneric_2016}\textup{)} 
Let $D$ be a holomorphic distribution on $\mathbb C^n$ of rank $d$ satisfying 
\eqref{invertible assumption}.
Then, for each $N\in \mathbb N$, there exists $C_N\in \mathbb R$ such that 
for any holomorphic horizontal disc 
$f \colon \Delta \rightarrow \left( \mathbb C^n /K , D|_{\mathbb C^n /K}\right)$
with $f(0) \in 2^{N_0}\Delta^n \subset \mathbb C^n$ for some $N_0$ and
\begin{equation}
K := \bigcup^{\infty}_{N=1} 2^{N-1} \partial\Delta^d \times C_N\Delta^{n-d},
\end{equation}
we have $|f'(0)| < c_{N_0}$.
Here the constant $c_{N_0}$ only depends on $N_0$.
\end{lem}

\begin{proof}

By permutation of coordinates we can assume that $X_1:=(X_{ji})_{j,i=1,\ldots, d}$ satisfies 
$\det X_1\neq 0$ for each point of $\mathbb C^n$.
Furthermore, by the condition \eqref{invertible assumption}, we can assume that $X$ has the following form
\begin{equation}
X = \bordermatrix{~ & d & n-d \cr
                  \,d & \Pi_1 & \Pi_2 \cr
                  n-d & 0 & \Pi_4 \cr}.
\end{equation}
Then the dual frame $\omega_1,\ldots, \omega_n$ of $X_1,\ldots, X_n$ is given by
$(X^{-1})^t dz$ where $dz = (dz_1,\ldots, dz_n)^t$.
Note that \begin{equation}
(X^{-1})^t  := \bordermatrix{~ & d & n-d \cr
                  \,d & \pi_1 & 0 \cr
                  n-d & \pi_3 & \pi_4 \cr}
= \bordermatrix{~ & d & n-d \cr
                  \,d & (\Pi_1^{-1})^t & 0 \cr
                  n-d & -(\Pi_4^{-1})^t \Pi_2^t  (\Pi_1^{-1})^t  & (\Pi_4^{-1})^t  \cr}.
\end{equation}
Let $f = (f_1, \ldots, f_n)\colon \Delta \rightarrow 
\left( \mathbb C^n /K , D|_{\mathbb C^n /K}\right)$ be a holomorphic horizontal disc 
satisfying $f(0) \in 2^{N_0}\Delta^n \subset \mathbb C^n$ for some $N_0$.
Since $f$ is a horizontal disc, we have $f^* \omega_j = 0$ for each $j=d+1, \ldots, n$.
This yields
\begin{equation}\label{eq1}
\left( f_{d+1}',\ldots,  f_{n}'\right)^t
= -\left( (\pi_4 )^{-1}\pi_3\right)\circ f \left( f_1',\ldots,  f_d'\right)^t.
\end{equation}
Note that each entry of the matrix $\left( (\pi_4 )^{-1}\pi_3\right)$ is a holomorphic 
function on $\mathbb C^n$. 
Choose a constant $C_N$ such that \begin{equation}
C_N > 2^{N} + d \,2^{2N}\sup_{\substack{z\in 2^N \Delta, \\j,i = 1,\ldots, n-d}}|
\left( (\pi_4 )^{-1}\pi_3\right)_{ji}(z)|.
\end{equation}
Then, by the same argument given by Forstneric in \cite{Forstneric_2016}, we obtain the lemma.
\end{proof}

\begin{lem}\label{lem2}
\textup{(Generalization of }\cite[Theorem 3.1]{Forstneric_2016}\textup{)} 
Let $0<a_1<b_1<a_2<b_2<\cdots $ and $c_i>0$ be sequences of real numbers such that 
$\lim_{i\rightarrow \infty} a_i= \lim_{i\rightarrow \infty}b_i = \infty$.
Let $n>1$ be an integer and set
\begin{equation}
K := \bigcup_{i=1}^\infty \left( b_i\overline \Delta^d \setminus a_i \Delta^{d}\right) \times c_i\overline \Delta^{n-d} \subset \mathbb C^n.
\end{equation}
Then there exists a Fatou-Bieberbach domain $\Omega\subset \mathbb C^n \setminus K$.
\end{lem}

\begin{example}
For each $n\geq 3$, there exists a holomorphic bracket generating distribution 
$D$ on $\mathbb C^n$ such that 
$(\mathbb C^n, D)$ is Kobayashi hyperbolic.
\end{example}

\section{Kobayashi hyperbolic Homogeneous manifolds}\label{homogeneous manifolds}
\label{homogeneous}
\subsection{Lie algebra structures}\label{graded Lie algebra}
Let $(M,D)$ be a Kobayashi hyperbolic $G$-homogeneous complex manifold 
with an invariant HBGD {with a linear semisimple lie group $G$.}
Let $V$ denote the isotropy subgroup of $G$ at $x_0\in M$ 
so that $M$ is biholomorphic to $G/V$.
\begin{lem}
$V$ is compact. 
\end{lem}
\begin{proof}
Since $G$ has finite number of components and acts on $M$ almost effectively, 
without loss of generality we may assume that $G\subset \text{Aut}(M,D)$ and 
$G$ is connected.
Since a connected semisimple Lie subgroup with finite center in a Lie group 
is closed (cf. \cite[Proposition 6.1.]{Onishchik_Vinberg}),
$G$ is closed in $\text{Aut}(M,D)$ and hence 
$V= G\cap \text{Isot}(M,D)$ is compact where $\text{Isot}(M,D)$ denotes 
the isotropy subgroup of $\text{Aut}(M,D)$ at $x_0$.
\end{proof}

Let $\mathfrak g$ denote the Lie algebra of $G$ and $\mathfrak v$ the Lie algebra of $V$ in $\mathfrak g$.
Then, there is a vector subspace $\mathfrak m$ of $\mathfrak g$ such that   
\begin{equation}\label{v+m}
\mathfrak g = \mathfrak v \oplus \mathfrak m \quad \text{ and } \quad
[\mathfrak v, \mathfrak m] \subset \mathfrak m.
\end{equation}
The tangent space of $M$ at $x_0$ can be identified with $\mathfrak g/\mathfrak v = \mathfrak m$.
Let $\mathfrak g_1^{\mathbb R}$ denote the subspace of $\mathfrak m$ corresponding to $D_{x_0}$.
Define 
\begin{equation}
\mathfrak g_2^{\mathbb R} 
= [\mathfrak g_1^{\mathbb R}, \mathfrak g_1^{\mathbb R}]/ (\mathfrak v + \mathfrak g_1^{\mathbb R}),\ldots, 
\mathfrak g_k^{\mathbb R} 
= [\mathfrak g_1^{\mathbb R}, \mathfrak g_{k-1}^{\mathbb R}]/ (\mathfrak v + \mathfrak g_1^{\mathbb R}+\cdots+\mathfrak g_{k-1}^{\mathbb R}).
\end{equation}
Then we have
\begin{equation}
\mathfrak g =  \mathfrak v \oplus \mathfrak g_1^{\mathbb R} \oplus \cdots \oplus \mathfrak g_k^{\mathbb R}.
\end{equation}
Since $M=G/V$ is a complex manifold and $D$ is $G$-invariant with compact $V$,
we may assume that 
there is an endomorphism $j\colon \mathfrak g\rightarrow \mathfrak g$ 
induced from the complex structure of $M$,
such that 
\begin{eqnarray}
&&\label{complex} j\mathfrak v =0, j^2x = -x, \\
&& \label{j-invariant} \text{Ad}_v jx = j\circ \text{Ad}_v x, \\
&& \label{complex structure} [jx,jy]= [x,y] + j[jx,y] + j[x,jy], \\
&& \label{v-invariant} \text{Ad}_v \mathfrak g_1^{\mathbb R} \subset \mathfrak g_1^{\mathbb R}, 
\end{eqnarray}
where $x,y\in \mathfrak m$ and $v\in V$.
Extend $j$ complex linearly to the complexification $\mathfrak g^{\mathbb C}$ of $\mathfrak g$,
and denote by $\mathfrak g^+$, $\mathfrak g^-$ the eigenspaces of $j$ 
with eigenvalues $\sqrt{-1}$, $-\sqrt{-1}$ respectively.
Then $G\times_V\mathfrak g^+$ and $G\times_V\mathfrak g^-$ represent the holomorphic tangent bundle and anti-holomorphic tangent bundle of $M$ respectively.
Let $\mathfrak g_1$ denote the subspace in $\mathfrak g^+$ corresponding to $D_{x_0}$.
Then $\frak g_1 = \{x-\sqrt{-1}jx \in \frak g^{\mathbb C}: x\in \frak g_1^{\mathbb R}\}$.
Define 
\begin{equation}
\mathfrak g_2 := [\mathfrak g_1, \mathfrak g_1]/(\mathfrak v^{\mathbb C} + \mathfrak g_1), \dots , \mathfrak g_k := [\mathfrak g_1, \mathfrak g_{k-1}]/(\mathfrak v^{\mathbb C} + \mathfrak g_1+\dots + \mathfrak g_{k-1}).
\end{equation}
Then \begin{equation}
\mathfrak g^+ = \mathfrak g_1 \oplus \cdots \oplus \mathfrak g_k.
\end{equation}
Similarly, let $\mathfrak g_{-1}\{x+\sqrt{-1}jx \in \frak g^{\mathbb C}: x\in \frak g_1^{\mathbb R}\}$ and construct $\frak g_{-2},\ldots, \mathfrak g_{-k}$. Set 
\begin{equation}\label{graded}
\begin{aligned}
\mathfrak g^- &= \mathfrak g_{-1} \oplus \cdots \oplus \mathfrak g_{-k}, \text{ and }\\
\mathfrak g^\CC &= \mathfrak g_{-k}\oplus \cdots 
\oplus\mathfrak g_{-1}\oplus\mathfrak g_0\oplus 
\mathfrak g_1 \oplus \cdots \oplus \mathfrak g_k
\end{aligned}
\end{equation}
where $\mathfrak g_0:=\mathfrak v^{\mathbb C}$.
For the notational convention, denote
\begin{equation}
\mathfrak g_{\leq l} := \sum_{j\leq l}\mathfrak g_j \
 \text{ and }\ \mathfrak g_{\geq l} := \sum_{\j\geq l}\mathfrak g_j.
\end{equation}

\subsection{Proof of Theorem \ref{main theorem} }
This section is a continuation of Section \ref{graded Lie algebra}.
We use the same notation. 
Since $(M,D)$ is Kobayashi hyperbolic, we obtain
\begin{equation}\label{No complex line}
[x,jx]\neq 0 \text{ for all } x\in \mathfrak g_1^{\mathbb R}.
\end{equation}
If not, there is a holomorphic mapping from the complex plane 
since $\RR x+ \RR jx$ is a complex Lie subalgebra of $\frak g$.
This contradicts Corollary \ref{No complex action}.

Since $\frak g$ is semisimple, 
there are simple Lie algebras $\frak s^1,\ldots, \frak s^N$ such that 
 $\frak g = \frak s^1\oplus\cdots\oplus\frak s^N$.
Note that $[\frak s^l, \frak s^m]=0$ for any $l,\, m$ with $l\neq m$.
Let $\frak s^l_1= \frak s^l\cap\frak g_1^\RR$.
Then $\frak s^l_1$ is non-empty and bracket generating in $\frak s^l$.
We obtain that $j(\frak s^l_1)\subset \frak s^l_1$.
Otherwise, there exists $x\in \frak s^l_1$ such that $[x,jx]=0$ 
and this contradicts \eqref{No complex line}.
By the bracket generating property of $\frak s^l_1$ in $\frak s^l$ 
and \eqref{complex structure}, 
we can get 
\begin{equation}
j(\frak s^l)\subset \frak s^l \text{ for each } l=1,\ldots, N.
\end{equation}
Suppose that there exists $l$ such that $\frak s^l$ is of compact type.
Let $G_l$ be a simple Lie group corresponding to $\frak s^l$
and $V_l=G_l\cap V$.
Then the $G_l$-orbit in $M$ at $x_0 (=G_l/V_l)$ with holomorphic distribution $G_l\times_{V_l} \frak s^l_1$
is a Kobayashi hyperbolic compact homogeneous complex submanifold in $M$.
This implies that $\text{Aut}(G_l/V_l)$ is a complex Lie group 
which contradicts Proposition \ref{No complex action}.
Hence, we have proved that $\frak g$ is a semisimple Lie algebra of non-compact type.

Let $\mathfrak k$ be a maximal compact subalgebra of $\mathfrak g$ containing $\mathfrak v$ 
and $\mathfrak g = \mathfrak k+\mathfrak q$ a Cartan decomposition with respect to $\mathfrak k$.
Let \begin{equation}
\mathfrak g^{\mathbb C} = \mathfrak g_{-k} +\cdots + \mathfrak g_{-1} + \mathfrak g_0 + \mathfrak g_1 +\cdots + \mathfrak g_k
\end{equation}
be the decomposition induced by $D$. Note that $\mathfrak g_{\leq 0}$ is a Lie subalgebra of $\mathfrak g^{\mathbb C}$.

Let $G_{\leq 0}$ denote the connected Lie subgroup 
corresponding to the Lie subalgebra $\mathfrak g_{\leq 0}$ in $G^{\mathbb C}$.
Let $N_{G^{\mathbb C}}(G_{\leq 0}) = \{ g\in G^{\mathbb C} : 
\text{Ad}_g (\frak g_{\leq 0})\subset \frak g_{\leq 0}\}$
be the normalizer of $G_{\leq 0}$. Note that $N_{G^{\mathbb C}}(G_{\leq 0})$ is a closed Lie subgroup of $G^{\mathbb C}$.
Since the Lie algebra of $N_{G^{\mathbb C}}(G_{\leq 0})$ is given by
$\{ X\in \frak g^{\mathbb C} : [X, \frak g_{\leq 0}]\subset \frak g_{\leq 0}\}
= \frak g_{\leq 0}$, it follows that $G_{\leq 0}$ is the identity component of 
$N_{G^{\mathbb C}}(G_{\leq 0})$.
Hence $G_{\leq 0}$ is a closed algebraic subgroup of $G^{\mathbb C}$
and $G^{\mathbb C}/G_{\leq 0}$ is a complex homogeneous manifold.

Let $G_u$ denote the compact real form of $G^{\mathbb C}$ with respect to $G$, 
i.e., the Lie algebra of $G^u$ is given by 
$\mathfrak g_u = \mathfrak k + \sqrt{-1} \mathfrak q$.
By Lemma \ref{theta_invariant} and the fact that $\frak g_{\leq 0}\cap \frak g = \frak v$, 
we obtain that $\frak g_u\cap \frak g_{\leq 0} = \frak v$ and hence 
\begin{equation}
\dim_\CC (G^u\text{-orbit at }eG_{\leq 0}) = \dim_\CC (G^\CC/G_{\leq 0}).
\end{equation}
It implies that the $G^u$-orbit of $eG_{\leq 0}$ is open. 
Besides, since the $G^u$-orbit is compact,
it is equal to $G^\CC/G_{\leq 0}$. 
In particular, $G^\CC/G_{\leq 0}$ is compact and hence $G_{\leq 0}$ is a parabolic subgroup
(cf. \cite[Chaper 3]{Akhiezer}).
By a similar procedure, we may conclude that $G_{\geq 0}$ is a parabolic subgroup.
Since $\mathfrak g_{\leq 0}\cap \mathfrak g_{\geq 0}
=\mathfrak g_0 = \mathfrak v^{\mathbb C}$,
there exists a Cartan subalgebra of $\mathfrak g^\CC$ contained in $\frak v^\CC$
(cf.  Corollary 2.1.3 in \cite{Huck_Fels_Wolf}). 
This implies that 
$\widetilde M$, the $G$-orbit of $eG_{\leq 0}$ in $G^\CC/G_{\leq 0}$, is a canonical flag domain.
Since flag domains are simply connected, the isotropy subgroup $G_{\leq 0}\cap G$ is connected
and hence it is the identity component of $V$.
This implies that 
$\widetilde M$ is the universal covering of $M=G/V$.

Let $g = \mathfrak k' + \mathfrak q'$ be a Cartan decomposition 
with corresponding Cartan involution $\theta'$.
It is known that $\theta'$ is an inner automorphism of $\mathfrak g$ 
if and only if $k'$ contains a Cartan subalgebra of $\mathfrak g$.
Let $\frak t$ denote $\frak t^\CC\cap \frak g$.
Note that $\frak t\subset \frak v\subset \frak k$.
Let $\theta$ be the Cartan involution with respect to $\frak k$, 
that is, $\theta(k+q)= k-q$ for all $k\in \frak k$ and $q\in \frak q$.
Since $\frak k$ contains a Cartan subalgebra of $\frak g$,
$\theta$ is an inner automorphism of $\frak g$.
This implies that $\theta$ can be expressed as $\text{Ad}_{\exp \xi}$, 
where $\xi$ is an element
of the center of $\frak k$. Hence $\xi\in \frak t\subset \frak v$.
By \eqref{v-invariant}, we obtain $\theta(\frak g_1^\RR) = \frak g_1^\RR$ and 
in particular, $\frak g_1^\RR = \frak g_1^\RR\cap \frak k + \frak g_1^\RR\cap \frak q$.

Suppose that $\mathfrak k_1 := \mathfrak k \cap \mathfrak g_1^{\mathbb R} \neq \{0\}$,
and define a Lie subalgebra 
$$
\mathfrak k' := \mathfrak v \oplus \mathfrak k_1 \oplus \mathfrak k_2 \oplus \cdots \oplus \mathfrak k_k
$$
where $\mathfrak k_2 = [\mathfrak k_1, \mathfrak k_1]/\mathfrak v,  \ldots, \mathfrak k_j=[\mathfrak k_1, \mathfrak k_{j-1}]/(\mathfrak v + \mathfrak k_0+\cdots+\mathfrak k_{j-1})$.
Let $K'$ denote the connected Lie subgroup of $G$ with respect to the Lie algebra $\mathfrak k'$.
Then $K'\times_V \frak k_1$ is a HBGD and 
$(K'/V, K'\times_V \frak k_1)$ is Kobayashi hyperbolic.
However, since $(K'/V, K'\times_V \frak k_1)$ is a compact complex manifold such that 
$\text{Aut}(K'/V, K'\times_V \frak k_1)$ acts transitively on $K'/V$, 
it cannot be Kobayashi hyperbolic.
Hence $\mathfrak k_1$ should be $\{0\}$.
This implies that $\mathfrak g_1^{\mathbb R} \subset \mathfrak q$.
Since the superhorizontal distribution on the canonical flag domain is 
the unique invariant HBGD contained in $G\times_V\frak q$, we obtain that 
$\frak g_1^\RR$ is a subalgebra which induces the superhorizontal distribution. \hfill$\Box$

\begin{lem}\label{theta_invariant}
$\frak g_{\geq 0}$ and $\frak g_{\leq 0}$ are $\theta$-invariant
where $\theta: k + q\mapsto k-q$ is a Cartan involution of $\frak g$ 
with $k\in \frak k$ and $q\in \frak q$.
\end{lem}
\begin{proof}
If not, there exists an element $x\in \frak g_{>0}$ such that $\theta(x)\in \frak g_{<0}$ 
since $\frak g_0 = \frak v^{\mathbb C}\subset \frak k^{\mathbb C}$ 
is $\theta$-invariant.
Denote $x= x_{\frak k} - \sqrt{-1}jx_{\frak k} + x_{\frak q} - \sqrt{-1} jx_{\frak q}$
with $x_{\frak k}\in \frak k$ and $x_{\frak q}\in \frak q$.
Let $\sigma$ be a complex conjugation with respect to $\frak g$ in $\frak g^{\mathbb C}$.
Since $\sigma (\frak g_{<0}) = \frak g_{>0}$, it follows that 
$\sigma \theta(x) =  x_{\frak k} + \sqrt{-1}jx_{\frak k} - x_{\frak q} 
- \sqrt{-1} jx_{\frak q}\in \frak g_{>0}$. 
Since $x+\sigma\theta(x)=2x_{\frak k}  - 2\sqrt{-1} jx_{\frak q}\in \frak g_{> 0}$, we conclude that 
$x_{\frak k}-\sqrt{-1} jx_{\frak q}$ is a $\sqrt{-1}$-eigenvector with respect to $j$.
This implies that $x_{\frak q} - x_{\frak k} + \sqrt{-1} j(x_{\frak q} - x_{\frak k})=0$
and in particular, $x_{\frak q} - x_{\frak k}=0$. Finally, it follows that $x_{\frak q}=x_{\frak k}=0$.
\end{proof}

\subsection{Miscellanea}
In this section, we assume that $G$ is a Lie group without 
semisimple condition and $V$ is compact. 
Under this condition the Lie algebra decompositions given in Section \ref{graded Lie algebra} hold.
Since $[\mathfrak g_0, \mathfrak g_1]\subset \mathfrak g_1$,  
we obtain $[\mathfrak g_0, \mathfrak g_\ell]
\subset \mathfrak g_{\leq \ell}\cap \frak g_{\geq 0}$ for each $2\leq \ell\leq k$.
Since $D$ is a holomorphic distribution, one has
$$[C^\infty(D), C^\infty(T_M)] \subset C^\infty (D\oplus T_M).$$
Here $C^\infty(E)$ denotes the set of smooth sections of the vector bundle $E$.
Using this expression, we obtain
\begin{equation}
[\mathfrak g_1, \mathfrak g^-]\subset \mathfrak g_{\leq 1}.
\end{equation}

\begin{lem}
\begin{equation}\label{graded Lie algebra structure}
\begin{aligned}
&[\mathfrak g_i, \mathfrak g_{-\ell}]\subset \mathfrak g_{\leq i-\ell+1} 
\text{ when } i \geq \ell,\\
&[\mathfrak g_i, \mathfrak g_{-\ell}]\subset \mathfrak g_{\geq i-\ell -1} 
\text{ when } i\leq\ell,\\
&[\mathfrak g_i, \mathfrak g_{-i}] \subset \frak g_{-1}+\mathfrak g_0 +\frak g_1\,\,
\text{ and } \,\,[\frak g_i^\RR, j\frak g_i^\RR]\subset \frak g_0 + \frak g_1, \text{ for all } i>0.
\end{aligned}
\end{equation}
\end{lem}

\begin{proof}
It is enough to prove the first expression.
Since $[\mathfrak g_1, \mathfrak g_{-1}] \subset 
\mathfrak g^- + \mathfrak g_0 + \frak g_1$, 
by taking complex conjugation, 
we obtain that $[\mathfrak g_1, \mathfrak g_{-1}] \subset \mathfrak g_0 + \frak g_1$.
Suppose that for every $i< i'$, 
$[\mathfrak g_i, \mathfrak g_{-1}] \subset \mathfrak g_{\leq i}$.
Then $[\mathfrak g_{i'}, \mathfrak g_{-1}] 
= [\mathfrak g_1, [\mathfrak g_{i'-1}, \mathfrak g_{-1}]] + [\mathfrak g_{i'-1}, [\mathfrak g_1, \mathfrak g_{-1}]]\subset \mathfrak g_{\leq i'}$ and hence 
for every $i$, $[\mathfrak g_i, \mathfrak g_{-1}]\subset \mathfrak g_{\leq i}.$

Fix $i'$ and suppose that for every $i<i'$, one has $[\mathfrak g_i, \mathfrak g_{-\ell}]\subset \mathfrak g_{\leq i-\ell+1}$ for all $1\leq \ell \leq i$. 
Furthermore, suppose that for fixed $\ell'$ with $\ell'\leq i'$, one has $[\mathfrak g_{i'}, \mathfrak g_{-\ell}]\subset \mathfrak g_{\leq i'-\ell+1}$ for all $\ell <\ell'$.
Then, since $[\mathfrak g_{i'}, \mathfrak g_{-\ell'}] = [[\mathfrak g_{i'}, \mathfrak g_{-\ell'+1}], \mathfrak g_{-1}] + [[\mathfrak g_{i'}, \mathfrak g_{-1}], \mathfrak g_{-\ell'+1}]$, we obtain that 
$[\mathfrak g_{i'}, \mathfrak g_{-\ell'}]\subset \mathfrak g_{\leq i'-\ell'+1}$
and the lemma is proved.
\end{proof}

\begin{lem}
Let $\mathfrak r$ be an abelian ideal of $\mathfrak g$.
Then 
\begin{equation}\label{effective}
\mathfrak r \cap \mathfrak v =0.
\end{equation}
\end{lem}
\begin{proof}
Let $w\in \mathfrak r \cap \mathfrak v$.
Then $\text{ad}_w^2 = 0$ on $\mathfrak g^\CC$.
Since $\mathfrak v$ is compact, 
we can consider $\text{ad}_w$ as a skew-symmetric matrix.
Since every skew-symmetric matrix is diagonalizable over $\mathbb C$, 
$\text{ad}_w$ is a semisimple element and hence $\text{ad}_w=0$.
By almost effectiveness we obtain $w=0$.
\end{proof}
\begin{lem}\label{bracket}
Let $\mathfrak r$ be an abelian ideal in $\mathfrak g$.
Then 
\begin{equation}\label{abelian ideal}
[\mathfrak g, \mathfrak g]\cap \mathfrak r = [\mathfrak r, \mathfrak g].
\end{equation}
\end{lem}

\begin{proof}
We will prove it by the induction with respect to the dimension of $\mathfrak g$.
If $\dim \mathfrak g =1$, then the statement is trivial.
Suppose the statement to be true for every Lie algebra $\mathfrak g$ of dimension less than or equal to $k$
and every abelian ideal $\mathfrak r$ in $\mathfrak g$.
Now let $\dim \mathfrak g = k+1$ and $\mathfrak r\subset \mathfrak g$ be an abelian ideal.
If $[\mathfrak g, \mathfrak r]=0$, then \eqref{abelian ideal} is true.
If $[\mathfrak g, \mathfrak r]=\mathfrak r$, then 
$\mathfrak r= [\mathfrak g, \mathfrak r]\subset [\mathfrak g,\mathfrak g]\cap \mathfrak r
\subset \mathfrak r$ and again the statement holds.
Assume that $\mathfrak r':=[\mathfrak g, \mathfrak r]\subsetneqq \mathfrak r$.
Then $\mathfrak r'$ is an ideal in $\mathfrak g$ 
and $\mathfrak r/\mathfrak r'$ is a non-trivial abelian ideal in $\mathfrak g/\mathfrak r'$.
Since $\dim \mathfrak g/\mathfrak r'\leq k$, we obtain 
$[\mathfrak g/\mathfrak r', \mathfrak g/\mathfrak r']\cap \mathfrak r/\mathfrak r'
= [\mathfrak g/ \mathfrak r', \mathfrak r/ \mathfrak r']=0$ and this implies \eqref{abelian ideal}.
\end{proof}

\begin{prop}
Let $\frak r$ be an abelian ideal in $\mathfrak g$.
\begin{enumerate}
\item
If $\frak r\,\cap \,\mathfrak g_1^\RR \neq \{0\}$,
then there exists an equivariantly embedded bounded homogeneous domain of type I
tangential to $D$.
\item 
If $\frak r\,\cap \,\mathfrak g_1^\RR =\{0\}$,
then $\frak r\cap \frak g^\RR_l=\{0\} \text{ for any } l= 1,\ldots, k$.
\end{enumerate}
\end{prop}
\begin{proof}
$(1): $
Let $\frak r_1$ denote $\mathfrak r\cap \mathfrak g_1^\RR$.
Then $\frak r_1\cap j\frak r_1=\{0\}$.
If not, there are $x,\, y \in \frak r_1$ such that 
$x=jy$ and hence $[x, jx] = [x,y]=0$.
However, this contradicts \eqref{No complex line}.
Furthermore $\frak r_1 + j\frak r_1$ is a subalgebra in $\mathfrak g$
since for any $x,y\in \frak r_1$, $[x,y]=0$, one has
$[x,jy]\in \frak r\cap (\frak g_0+\frak g_1) = \frak r_1$ 
(by \eqref{effective} and \eqref{graded Lie algebra structure}),
 and $[jx,jy]\in j\frak r_1$ (by \eqref{j-invariant}).
Let $R$ be the connected Lie subgroup of $G$ corresponding to $\frak r_1 + j\frak r_1$.
Then the $R$-orbit in $M$ is an equivalently embedded 
bounded homogeneous domain tangential to $D$. \\
$(2): $ It follows from Lemma \ref{bracket}.
\end{proof}

\end{document}